\documentclass[12pt]{article}
\usepackage{amsmath,amssymb,color,amscd}

\def\seq#1#2#3{#1_{#2},\,\ldots,#1_{#3}}

\def\w{\widetilde}
\def\h{\widehat}

\def\vv{{\underline{v}}}
\def\nuv{{\underline{\nu}}}

\def\tt{{\underline{t}}}

\def\mm{\underline{m}}
\def\MM{\underline{M}}

\def\1{\underline{1}}

\def\P{\mathbb P}

\def\Z{\mathbb Z}
\def\Q{\mathbb Q}
\def\C{\mathbb C}

\def\OO{{\mathcal O}}

\newtheorem{theorem}{Theorem}

\newenvironment{definition}
{\smallskip\noindent{\bf Definition\/}:}{\smallskip\par}

\newenvironment{proof}
{\noindent{\bf Proof\/}.}{{ $\Box$}\smallskip\par}

\title{Equivariant Poincar\'e series and topology of valuations\footnote{Math. Subject Class. 14B05, 13A18,
14R20, 16W70.
Keywords: finite group actions, Poincar\'e series,  plane valuations, equivariant topology}
}

\author{
A.~Campillo,
\and F.~Delgado,\thanks{Supported by the grants
MTM2012-36917-C03-01 / 02
(both grants with the help of FEDER Program).} \and S.M.~Gusein-Zade
\thanks{
Supported by the
grants RFBR--13-01-00755, NSh--5138.2014.1.
} }

\date{}
\begin{document}
\def\eps{\varepsilon}

\maketitle

\begin{abstract}
The equivariant with respect to a finite group action Poincar\'e series of a collection
of $r$ valuations was defined earlier as a power series in $r$ variables with the coefficients
from a modification of the Burnside ring of the group. Here we show that (modulo simple exceptions)
the equivariant Poincar\'e series determines the equivariant topology of the collection
of valuations.
\end{abstract}

%%%%%%%%%%%%%%%%%%%%%%%%%%%%%%%%%%%%%%%%%%%%%%%%%%%%%%%%%%%%%%%%%
\section{Introduction}\label{sec1}
%%%%%%%%%%%%%%%%%%%%%%%%%%%%%%%%%%%%%%%%%%%%%%%%%%%%%%%%%%%%%%%%%

A definition of the Poincar\'e series of a multi-index filtration was first given in \cite{CDK}
(for filtrations defined by collections of valuations). It is a formal power series in several variables
with integer coefficients, i.e., an element of the ring $\Z[[t_1, \ldots, t_r]]$.
In \cite{Duke} it was shown that, for
the filtration defined by the curve valuations corresponding to the irreducible components of a
plane curve singularity, the Poincar\'e series coincides with the Alexander polynomial in several
variables of the corresponding algebraic link: the intersection of the curve with a small sphere
in $\C^2$ centred at the origin. This relation was obtained by a direct computation of the both
sides in the same terms. Up to now there exist no conceptual proof of it.
The Alexander polynomial in several variables of an algebraic link (and therefore the Poincar\'e series
of the corresponding collection of valuations) determines the topological type of the corresponding plane
curve singularity.
 In \cite{IJM} the definition of the Poincar\'e series was reformulated in terms of an integral
with respect to the Euler characteristics (over an infinite dimensional space).

The desire to understand deeper this relation led to attempts to find an equivariant version of it
(for actions of a finite group $G$)
and thus to define equivariant versions of the Poincar\'e series and of the Alexander polynomial.
Some equivariant versions of the monodromy zeta-function (that is of the Alexander polynomial
in one variable) were defined in \cite{GLM1} and \cite{GLM2}. Equivariant versions of the
Poincar\'e series were defined in \cite{Equiv1}, \cite{Equiv2} and \cite{Equiv3}.

In some constructions of equivariant analogues of invariants (especially those related to the Euler
characteristic) the role of the ring of integers $\Z$ (where the Euler characteristic takes values)
is played by the Burnside ring $A(G)$ of the group $G$. Therefore it would be attractive to define
equivariant versions of the Poincar\'e series as elements of the ring $A(G)[[t_1, \ldots, t_r]]$
(or of a similar one). The equivariant versions of the monodromy zeta functions defined in \cite{GLM1}
and \cite{GLM2} are formal power series with the coefficients from $A(G)\otimes \Q$ and $A(G)$
respectively.

In \cite{Equiv1} the equivariant Poincar\'e series was defined as an element of the ring $R_1(G)[[t_1, \ldots, t_r]]$
of formal power series in $t_1$, \dots, $t_r$ with the coefficients from the subring $R_1(G)$ of
the ring $R(G)$ of complex representations of the group $G$ generated by the one-dimensional representations.
This Poincar\'e series turned out to be useful for some problems: see, e.g., \cite{FAA2}, \cite{Nem}.
However, it seems to be rather ``degenerate'', especially for non-abelian groups.

In \cite{Equiv2} the $G$-equivariant Poincar\'e series $P^G_{\{\nu_i\}}$ of a collection of valuations (or order
functions) $\{\nu_i\}$ was not in fact a series, but an element of the Grothendieck ring of so called locally
finite $(G,r)$-sets. This Grothendieck ring was rather big and complicated, the Poincar\'e series $P^G_{\{\nu_i\}}$
was rather complicated as well and contained a lot of information about the valuations and the $G$-action.
In particular, for curve and divisorial valuations on the ring  $\OO_{\C^2,0}$ of functions in two
variables the information contained in this
Poincar\'e series was (almost) sufficient to restore the action of $G$ on $\C^2$ and the $G$-equivariant
topology of the set of valuations: \cite{Ark}.

In \cite{Equiv3} the equivariant Poincar\'e series $P^G_{\{\nu_i\}}(t_1,\ldots, t_r)$ was defined as an element
of the ring $\widetilde{A}(G)[[t_1,\ldots, t_r]]$ of formal power series in the variables $t_1,\ldots, t_r$
with the coefficients from a certain modification $\widetilde{A}(G)$ of the Bunside ring $A(G)$ of the
group $G$.
A simple reduction of this Poincar\'e series is an element of the ring $A(G)[[t_1,\ldots, t_r]]$. Thus it is
somewhat close to the (``idealstic") model discussed above. However, in order to define the equivariant Poincar\'e series
of this form, it was necessary to lose quite a lot of information about the individual valuations from the
collection. (It is possible to say that one used averaging of the information over the group.) Thus it was not
clear how much information does it keep.

Here we discuss to which extend the $G$-equivariant Poincar\'e series from \cite{Equiv3} determines the topology
of a set of plane valuations. The answer is rather similar to the one in \cite{Ark}, however reasons for that
(and thus the proofs) turn out to be much more involved.

The $G$-equivariant Poincar\'e series $P^G_{\{\nu_i\}}$ considered in \cite{Equiv2} depends essentially on the
set of valuations defining the filtration. In particular, the substitution of one of them (say, $\nu_i$) by its
shift $a^*\nu_i$, $a\in G$, changes the $G$-equivariant Poincar\'e series $P^G_{\{\nu_i\}}$. The Poincar\'e series
$P^G_{\{\nu_i\}}(\tt)$ considered in \cite{Equiv3} depends not on the valuations $\nu_i$ themselves, but on
their $G$-orbits. The substitution of one of them by its shift does not change the $G$-equivariant
Poincar\'e series $P^G_{\{\nu_i\}}(\tt)$. Therefore this series cannot determine the $G$-topology of a
collection of divisorial and/or of curve valuations on $\OO_{\C^2,0}$ in the form defined in \cite{Ark}.
One has to modify this notion a little bit.

Assume first that we consider sets of curve valuations. Let $\{C_i\}_{i=1}^r$ and $\{C'_i\}_{i=1}^r$ be
two collections of branches (that is of irreducible plane curve singularities) in the complex plane $(\C^2,0)$
with an action of a finite group $G$. We shall say that these collections are {\em weakly $G$-topological
equivalent} if there exists a $G$-invariant germ of a homeomorphism $\psi:(\C^2,0)\to(\C^2,0)$ such that
for each $i=1, \ldots, r$ one has $\psi(C_i)=a_iC'_i$ with an element $a_i\in G$ (i.e if the image of the
$G$-orbit of the branch $C_i$ coincides with the $G$-orbit of the branch $C'_i$). To formualate an analogue
of this definition for collections of divisorial valuations, one can describe a divisorial valuation $\nu$
on $\OO_{\C^2,0}$ by a pair of curvettes intersecting the corresponding divisor (transversally) at different
points. Two collections of divisorial valuations $\{\nu_i\}_{i=1}^r$ and $\{\nu'_i\}_{i=1}^r$ described
by the corresponding collections of curvettes $\{L_{ij}\}^r_{i=1,j=1,2}$ and $\{L'_{ij}\}^r_{i=1,j=1,2}$
respectively are weakly $G$-topologically equivalent if there exists a $G$-invariant germ of a homeomorphism
$\psi:(\C^2,0)\to(\C^2,0)$ such that for each $i=1, \ldots, r$ one has $\psi(L_{ij})=a_iL'_{ij}$
for $j=1,2$ and an element $a_i\in G$.

One has an obvious analogue of Theorem~2.9 from \cite{Ark}. This means that, for a fixed representation
of the group $G$ on $\C^2$, the weak topology of a collection of curve or/and divisorial valuations on
$\OO_{\C^2, 0}$ is determined by the $G$-resolution graph $\Gamma^G$ of the collection (where not individual
branches or/and divisors, but their orbits are indicated) plus the correspondence between the tails of this
graph emerging from special points of the first component of the exceptional divisor with these special points
(see below).

%%%%%%%%%%%%%%%%%%%%%%%%%%%%%%%%%%%%%%%%%%%%%%%%%%%%%%%%%%%%%
\section{Equivariant Poincar\'e series}\label{sec2}
%%%%%%%%%%%%%%%%%%%%%%%%%%%%%%%%%%%%%%%%%%%%%%%%%%%%%%%%%%%%%

Let us briefly recall the definition of the $G$-equivariant Poincar\'e series
$P^G_{\{\nu_i\}}(t_1,\ldots, t_r)$
of a collection of order functions on the ring $\OO_{V,0}$ of germs of functions on $(V,0)$
and the equation for it in terms of a $G$-equivariant resolution of curve or/and divisorial plane
valuations which will be used here.

\begin{definition}
A finite {\it equipped} $G$-{\it set} is a pair $\w{X}=(X, \alpha)$ where:
\begin{enumerate}
\item[$\bullet$] $X$ is a finite $G$-set;
\item[$\bullet$] $\alpha$ associates to each point $x\in X$ a
one-dimensional representation
$\alpha_x$ of the isotropy subgroup $G_x=\{a\in G : ax=x\}$ of the point $x$ so that,  for
$a\in G$, one has $\alpha_{ax}(b)=\alpha_x(a^{-1}ba)$, where
$b\in G_{ax}=a G_x a^{-1}$.
\end{enumerate}
\end{definition}

Let $\w{A}(G)$ be the Grothendieck group of finite equipped $G$-sets.
The cartesian product defines a ring structure on it.
The class of an equipped $G$-set $\w X$ in the Grothendieck ring $\w{A}(G)$ will be
denoted by $[\widetilde{X}]$.
As an abelian group $\widetilde{A}(G)$ is freely generated by the classes of the
irreducible equipped $G$-sets $[G/H]_{\alpha}$ for all the conjugacy classes $[H]$
of subgroups of $G$ and for all one-dimensional representations $\alpha$ of $H$
(a representative of the conjugacy class $[H]\in \mbox{Conjsub\,}G$).

There is a natural homomorphism $\rho$ from the ring $\w{A}(G)$ to the Burnside rings $A(G)$ of the group $G$
defined by forgetting the one-dimensional representation corresponding to the points.
The reduction $\hat{\rho}:\w{A}(G)\to \Z$ is defined by forgetting the representations and the $G$-action.
There are natural pre-$\lambda$-structure on a rings $A(G)$ and $\w{A}(G)$
which give sense for the expressions of the form $(1-t)^{-[X]}$, $[X]\in A(G)$, and
$(1-t)^{-[\w{X}]}$, $[\w{X}]\in \w{A}(G)$ respectively:
see \cite{Equiv3}. Both $\rho$ and $\hat{\rho}$ are homomorphisms of pre-$\lambda$-rings.

Let $(V,0)$ be a germ of a complex analytic space
with an action of a finite group $G$ and let $\OO_{V,0}$ be the ring of germs of
functions on it. Without loss of generality we assume that the $G$-action on $(V,0)$
is faithful. The group $G$ acts on $\OO_{V,0}$ by $a^* f(z)=f(a^{-1}z)$ ($z\in V$, $a\in G$).
A valuation $\nu$ on the ring $\OO_{V,0}$
is a function $\nu: \OO_{V,0}\to \Z_{\ge 0}\cup \{+\infty\}$ such that:
\begin{enumerate}
\item[1)] $\nu(\lambda f) = \nu (f)$ for $\lambda\in \C^{*}$;
\item[2)] $\nu(f+g)\ge \min \{\nu(f), \nu(g)\}$;
\item[3)] $\nu(fg)= \nu(f)+\nu(g)$.
\end{enumerate}
A function $\nu:\OO_{V,0}\to \Z_{\ge 0}\cup \{+\infty\}$ which
possesses the properties 1)
and 2) is called an {\it order function}.

Let $\seq{\nu}1r$ be a collection of order functions
on
$\OO_{V,0}$.
It defines an $r$-index filtration on $\OO_{V,0}$:
$$
J(\vv) = \{h\in \OO_{V,0} : \nuv(h)\ge \vv\}\; ,
$$
where $\vv=(\seq v1r)\in \Z_{\ge 0}^{r}$,
$\nuv(h) = (\nu_1(h), \ldots, \nu_r(h))$ and
$\vv'=(\seq{v'}1r)\ge \vv'' =(\seq{v''}1r)$ if and only if
$v'_i\ge v''_i$ for all $i$.

Let $\omega_i:\OO_{V,0}\to \Z_{\ge 0}\cup
\{+\infty\}$ be defined by $\omega_i=\sum_{a\in G}a^* \nu_i$. The
functions $\omega_i$ are $G$-invariant (they are not, in general, order functions).
For an element $h\in\P\OO_{V,0}$, that is
for a function germ considered up to a constant factor, let $G_h$ be the isotropy
subgroup $G_h=\{a\in G: a^*h=\alpha_h(a)h\}$
and let $Gh\cong G/G_h$ be the orbit
of $h$ in $\P\OO_{V,0}$.
The correspondence $a\mapsto \alpha_h(a)\in\C^*$ determines
a one-dimensional representation $\alpha_{h}$ of the subgroup $G_h$.
Let $\widetilde{X}_h = [G/G_{h}]_{\alpha_h}$ be the element of
the ring $\widetilde{A}(G)$
represented by the $G$-set $Gh$ with the representation
$\alpha_{a^*h}$ associated to the point $a^*h\in Gh$ ($a\in G$).
The correspondence $h\mapsto \widetilde{X}_h$ defines a function
($\widetilde{X}$) on $\P\OO_{V,0}/G$ with values in $\widetilde{A}(G)$.
The {\it equivariant Poincar\'e series} $P^G_{\{\nu_i\}}(\tt)$ of the collection $\{\nu_i\}$ is defined by the equation
 \begin{equation}\label{main_definition}
 P^G_{\{\nu_i\}}(\tt)=\int_{\P\OO_{V,0}/G}
\widetilde{X}_h\tt^{\underline{\omega}(h)} d\chi\in
 \widetilde{A}(G)[[t_1, \ldots, t_r]]\; ,
 \end{equation}
where $\tt:= (\seq t1r)$,
$\tt^{\underline{\omega}(h)}=t_1^{\omega_1(h)}\cdot\ldots\cdot t_r^{\omega_r(h)}$, $t_i^{+\infty}$ should be regarded as 0.
The precise meaning of this integral see in \cite{Equiv3}.

Applying the reduction homomorphism $\rho:\widetilde{A}(G)\to A(G)$ to the Poincar\'e
series $P^G_{\{\nu_i\}}(\tt)$, i.e. to its coefficients, one gets the series
$\rho P^G_{\{\nu_i\}}(\tt)\in A(G)[[t_1, \ldots, t_r]]$, i.e. a power series
with the coefficients from the (usual) Burnside ring.
Applying the homomorphism $\widehat{\rho}:\widetilde{A}(G)\to \Z$
one gets the series $\widehat{\rho} P^G_{\{\nu_i\}}(\tt)\in \Z[[t_1, \ldots, t_r]]$.
 One has
$$
\widehat{\rho} P^G_{\{\nu_i\}}(\tt)=P_{\{a^*\nu_i\}}(t_1, \ldots, t_1, t_2, \ldots,
t_2,
\ldots, t_r, \ldots, t_r)\,,
 $$
 where $P_{\{a^*\nu_i\}}(\bullet)$ is the usual (non-equivariant) Poincar\'e series of the collection
 of $\vert G\vert r$ order functions $\{a^*\nu_1, a^*\nu_2,\ldots, a^*\nu_r\vert a\in G\}$
(each group of equal variables in $P_{\{a^*\nu_i\}}$ consists of $\vert G\vert$ of them).

Now assume that a finite group $G$ acts linearly on $(\C^2,0)$ and let $\nu_i$, $i=1, \ldots, r$,
be either a curve or a divisorial valuation on $\OO_{\C^2,0}$.
We shall write
$I_0=\{1,2,\ldots,r\}=I'\sqcup I''$, where $i\in I'$ if and only if the corresponding valuation
$\nu_i$ is a curve one. For $i\in I'$, let $(C_i,0)$ be the
plane curve defining the valuation $\nu_i$.

A $G$-{\it equivariant resolution} (or a $G$-{\it resolution} for short) of
the collection $\{\nu_i\}$ of valuations is a proper complex
analytic map $\pi: ({\cal X, D})\to (\C^2,0)$ from a
smooth surface ${\cal X}$ with a $G$-action such that:
\begin{enumerate}
\item[1)]
$\pi$ is an isomorphism outside of the origin in $\C^2$;
\item[2)]
$\pi$ commutes with the $G$-actions on ${\cal X}$ and on $\C^2$;
\item[3)]
the total transform $\pi^{-1}(\bigcup\limits_{i\in I',\, a\in G}
aC_i)$ of the curve $GC=G(\bigcup\limits_{i\in I'} C_i) $
is a normal crossing divisor on ${\cal X}$ (in particular,
the exceptional
divisor ${\cal D}=\pi^{-1}(0)$ is a normal crossing divisor as
well);
\item[4)]
for each branch $C_i$, $i\in I'$, its strict transform
$\widetilde{C}_i$ is a germ of a smooth curve transversal to the
exceptional divisor ${\cal D}$ at a smooth point $x$ of
it and is invariant with respect to the isotropy subgroup
$G_x=\{g\in G: gx=x\}$ of the point $x$;
\item[5)] for each $i\in I''$, the exceptional divisor ${\cal D}=\pi^{-1}(0)$
contains the divisor defining the divisorial valuation $\nu_i$.
\end{enumerate}
A $G$-resolution can be obtained by a $G$-invariant sequence of blow-ups of points.

The action of the group $G$ on the first component of the exceptional divisor
can either be trivial (this may happen only if $G$ is cyclic) or have fixed points
of (proper) subgroups of $G$. (If $G$ is abelian, these are the fixed points of $G$ itself.)
These points are called {\em special}.

Let ${\stackrel{\circ}{\cal D}}$ be the ``smooth part" of
the exceptional divisor ${\cal D}$ in the total transform
$\pi^{-1}(GC)$ of the curve $GC$, i.e., ${\cal D}$ itself minus
all the intersection points of its components
and all the intersection points with the components of the strict transform of the curve $GC$.
For $x\in {\stackrel{\circ}{\cal D}}$, let $\w{L}_x$ be a germ
of a smooth curve on ${\cal X}$ transversal to
${\stackrel{\circ}{\cal D}}$ at the point $x$ and invariant
with respect to the isotropy subgroup $G_x$ of the point $x$.
The image $L_x=\pi(\w{L}_x)\subset (\C^2,0)$ is called a {\it
curvette} at the point $x$. Let the curvette $L_x$ be given by an
equation $h_x=0$, $h_x\in\OO_{\C^2,0}$. Without loss of generality
one can assume that the function germ $h_x$ is $G_x$-equivariant.
Moreover we shall assume that the germs $h_x$ associated to different
points $x\in {\stackrel{\circ}{\cal D}}$ are choosen so that
$h_{ax}(a^{-1}z)/h_x(z)$ is a constant (depending on $a$ and $x$).

Let $E_\sigma$, $\sigma\in\Gamma$, be the set of all irreducible components of the
exceptional divisor $\cal {D}$ ($\Gamma$ is a $G$-set itself). For $\sigma$ and $\delta$
from $\Gamma$, let $m_{\sigma\delta}:=\nu_{\sigma}(h_x)$, where $\nu_{\sigma}$ is the
corresponding
divisorial valuation, $h_x$ is the germ defining the curvette at a point
$x\in E_{\delta}\cap {\stackrel{\circ}{\cal D}}$. One can show that the matrix $(m_{\sigma\delta})$
is minus the inverse matrix to the intersection matrix $(E_{\sigma}\circ E_{\delta})$
of the irreducible components of the exceptional divisor $\cal {D}$. For $i=1,\ldots,r$,
let $m_{\sigma i}:=m_{\sigma\delta}$, where $E_{\delta}$ is the component of $\cal {D}$
corresponding to the valuation $\nu_i$, i.e. either the component defining the valuation
$\nu_i$ if $\nu_i$ is a divisorial valuation (i.e. if $i\in I''$), or the component intersecting
the strict transform of the corresponding irreducible curve $C_i$ if $\nu_i$ is a curve valuation
(i.e. if $i\in I'$). Let
$\mm_{\sigma}:=(m_{\sigma 1},\ldots, m_{\sigma r})\in \Z_{\ge 0}^r$,
$M_{\sigma i}:=\sum_{a\in G}m_{(a\sigma)i}$,
$\MM_{\sigma}:=(M_{\sigma 1},\ldots,
M_{\sigma r})=\sum_{a\in G}\mm_{a\sigma}$.

Let $\widehat{\cal D}$ be the quotient ${\stackrel{\circ}{\cal D}}/G$ and let
$p: {\stackrel{\circ}{\cal D}} \to \h{\cal D}$ be the factorization map.
Let $\{\Xi\}$ be a stratification of the smooth curve
$\widehat{\cal D}$ such that:
\begin{enumerate}
\item[1)] each stratum $\Xi$ is connected;
\item[2)] for each point $\h{x}\in \Xi$ and for each point $x$
from its pre-image $p^{-1}(\h{x})$, the conjugacy class of the
isotropy subgroup $G_x$ of the point $x$ is the same, i.e., depends only on
the stratum $\Xi$.
\end{enumerate}
The condition 2)
is equivalent to say that the factorization map
$p: {\stackrel{\circ}{\cal D}} \to \h{\cal D}$ is a (non-ramified) covering
over each stratum $\Xi$.
The condition 1) implies that the inverse image in ${\stackrel{\circ}{\cal D}}$
of each stratum $\Xi$ lies in the orbit of one component $E_\sigma$ of the
exceptional divisor. The element $\MM_{\sigma}\in \Z_{\ge 0}^r$ depends only on the stratum $\Xi$ and
will be denoted by $\MM_{\Xi}$.

For a point $x\in {\stackrel{\circ}{\cal D}}$, let
$\widetilde{X}_x = [G/G_x]_{\alpha_{h_x}}\in \widetilde{A}(G)$.
The equipped $G$-set $\widetilde{X}_x$ is one and the same for all points $x$ from
the preimage of a stratum $\Xi$ and therefore it defines an element of $\widetilde{A}(G)$ which we shall
denote by $[G/G_{\Xi}]_{\alpha_{\Xi}}$.
In \cite[Theorem 1]{Equiv3} it was shown that
 \begin{equation}\label{ACampo}
 P^G_{\{\nu_i\}} (\tt) =\prod_{\Xi}
\left(1-\tt^{\MM_{\Xi}}\right)^{-\chi(\Xi)[G/G_{\Xi}]_{\alpha_{\Xi}}}\,.
 \end{equation}

%%%%%%%%%%%%%%%%%%%%%%%%%%%%%%%%%%%%%%%%%%%%%%%%%%%%%%%%%%%%
\section{Topology of plane valuations}\label{sec3}
%%%%%%%%%%%%%%%%%%%%%%%%%%%%%%%%%%%%%%%%%%%%%%%%%%%%%%%%%%%%

Let the complex plane $(\C^2, 0)$ be endowed by a faithful linear $G$-action and let $\{\nu_i\}_{i=1}^r$
be a collection of divisorial valuations on $\OO_{\C^2,0}$.

\begin{theorem}\label{theo1}
 The $G$-equivariant Poincar\'e series $P^G_{\{\nu_i\}}(\tt)$ of the collection $\{\nu_i\}$ of divisorial
 valuations determines the weak $G$-equivariant topology of this collection.
\end{theorem}

\begin{proof}
One has to use the following ``projection formula''. Let $I=\{i_1,\ldots, i_s\}$ be a subset of the set
$\{1, \ldots, r\}$ of the indices numbering the valuations. Then one has
$$
P^G_{\{\nu_i\}_{i\in I}}(t_{i_1},\ldots, t_{i_s})=
P^G_{\{\nu_i\}_{i=1}^r}(t_1,\ldots, t_r)_{\vert t_i=1\text{\ for\ }i\notin I}\,,
$$
i.e. the ($G$-equivariant) Poincar\'e series for a subcollection of valuations is obtained from the one
for the whole collection by substituting $t_i$  by $1$ for all $i$ numbering the valuations
which do not participate in the subcollection. (This equation is not valid for other types of valuations,
say, for curve ones: see the proof of Theorem~\ref{curve_top}).
The projection formula implies, in particular, that the $G$-equivariant Poincar\'e series $P^G_{\{\nu_i\}}(\tt)$ of
a collection of divisorial valuations determines the $G$-equivariant Poincar\'e series (in one variable) of
each individual valuation from it.

First we shall show that the Poincar\'e series $P^G_{\{\nu_i\}}(\tt)$ determines
the $G$-resolution graph of the collection of valuations.
It turns out that the necessary information about the $G$-equivariant resolution graph
can be restored from the $\rho$-reduction $\rho P^G_{\nu}(t)$ of the $G$-equivariant
Poincar\'e series $P^G_{\nu}(t)$ (i.e. the series from $A(G)[[t]]$ obtained by forgetting the one-dimensional
representations associated with the $G$-orbits). Therefore we shall start with considering it.

First let us prove the statement for one divisorial valuation.
The dual graph $\Gamma^G$ of the minimal $G$-equivariant resolution of a divisorial valuation $\nu$ looks like
in Fig.~\ref{fig1}. This means the following.

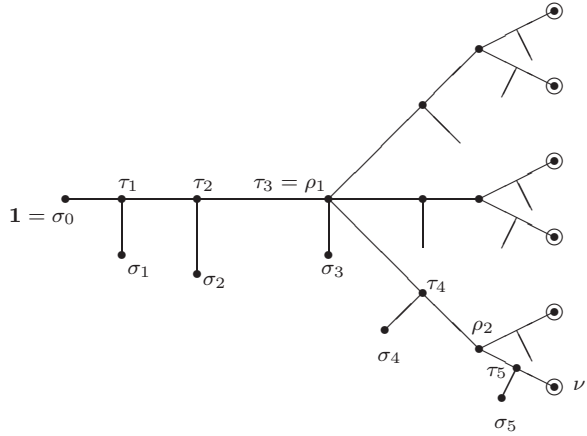
\begin{figure}
$$
\unitlength=0.50mm
\begin{picture}(120.00,90.00)(0,-30)
\thinlines
\put(-10,30){\line(1,0){70}}
\put(60,30){\circle*{2}}

\put(60,30){\line(1,1){40}}
\put(85,55){\line(1,-1){10}}
\put(85,55){\circle*{2}}
\put(100,70){\circle*{2}}

\put(100,70){\line(2,1){20}}
\put(120,80){\circle{4}}
\put(120,80){\circle*{2}}
\put(110,75){\line(1,-2){4}}
\put(100,70){\line(2,-1){20}}
\put(120,60){\circle{4}}
\put(120,60){\circle*{2}}
\put(110,65){\line(-1,-2){4}}

\put(60,30){\line(1,-1){40}}
\put(100,-10){\circle*{2}}
\put(85,5){\line(-1,-1){10}}
\put(85,5){\circle*{2}}
\put(86,6){{\scriptsize$\tau_4$}}
\put(75,-5){\circle*{2}}
\put(73,-13){{\scriptsize$\sigma_4$}}

\put(98,-5){{\scriptsize$\rho_2$}}
\put(100,-10){\line(2,1){20}}
\put(120,0){\circle{4}}
\put(120,0){\circle*{2}}
\put(110,-5){\line(1,-2){4}}
\put(100,-10){\line(2,-1){20}}
\put(120,-20){\circle{4}}
\put(120,-20){\circle*{2}}
\put(125,-21){{\scriptsize$\nu$}}
\put(110,-15){\line(-1,-2){4}}
\put(102,-17){{\scriptsize$\tau_5$}}
\put(110,-15){\circle*{2}}
\put(106,-23){\circle*{2}}
\put(104,-31){{\scriptsize$\sigma_5$}}

\put(60,30){\line(1,0){40}}
\put(85,30){\line(0,-1){13}}
\put(85,30){\circle*{2}}
\put(100,30){\circle*{2}}

\put(100,30){\line(2,1){20}}
\put(120,40){\circle{4}}
\put(120,40){\circle*{2}}
\put(110,35){\line(1,-2){4}}
\put(100,30){\line(2,-1){20}}
\put(120,20){\circle{4}}
\put(120,20){\circle*{2}}
\put(110,25){\line(-1,-2){4}}

\put(-10,30){\circle*{2}}
\put(5,30){\line(0,-1){15}}
\put(5,30){\circle*{2}}
\put(5,15){\circle*{2}}
\put(25,30){\line(0,-1){20}}
\put(25,30){\circle*{2}}
\put(25,10){\circle*{2}}
\put(60,30){\line(0,-1){15}}
\put(60,15){\circle*{2}}

\put(-25,24){{\scriptsize ${\bf 1}=\sigma_0$}}
\put(6.5,10){{\scriptsize$\sigma_1$}}
\put(26.5,7){{\scriptsize$\sigma_2$}}
\put(58,10){{\scriptsize$\sigma_3$}}
\put(4,33){{\scriptsize$\tau_1$}}
\put(24,33){{\scriptsize$\tau_2$}}
\put(40,33){\scriptsize{$\tau_3=\rho_1$}}
\end{picture}
$$
\caption{The dual equivariant resolution graph $\Gamma^G$ of the valuation $\nu$.}
\label{fig1}
\end{figure}

\begin{figure}
$$
\unitlength=0.50mm
\begin{picture}(120.00,40.00)(0,10)
\thinlines
\put(-10,30){\line(1,0){70}}
\put(60,30){\circle*{2}}

\put(82,33){{\scriptsize$\tau_4$}}
\put(82,12){{\scriptsize$\sigma_4$}}

\put(98,33){{\scriptsize$\rho_2$}}
\put(124,29){{\scriptsize$\nu$}}
\put(107,33){{\scriptsize$\tau_5$}}
\put(107,16){{\scriptsize$\sigma_5$}}

\put(60,30){\line(1,0){60}}
\put(85,30){\line(0,-1){13}}
\put(85,30){\circle*{2}}
\put(85,17){\circle*{2}}
\put(100,30){\circle*{2}}

\put(120,30){\circle{4}}
\put(120,30){\circle*{2}}
\put(110,30){\line(0,-1){8}}
\put(110,30){\circle*{2}}
\put(110,22){\circle*{2}}

\put(-10,30){\circle*{2}}
\put(5,30){\line(0,-1){15}}
\put(5,30){\circle*{2}}
\put(5,15){\circle*{2}}
\put(25,30){\line(0,-1){20}}
\put(25,30){\circle*{2}}
\put(25,10){\circle*{2}}
\put(60,30){\line(0,-1){15}}
\put(60,15){\circle*{2}}

\put(-25,24){{\scriptsize ${\bf 1}=\sigma_0$}}
\put(6.5,10){{\scriptsize$\sigma_1$}}
\put(26.5,7){{\scriptsize$\sigma_2$}}
\put(58,10){{\scriptsize$\sigma_3$}}
\put(4,33){{\scriptsize$\tau_1$}}
\put(24,33){{\scriptsize$\tau_2$}}
\put(50,33){\scriptsize{$\tau_3=\rho_1$}}
\end{picture}
$$
\caption{The dual resolution graph $\Gamma$ of the valuation $\nu$.}
\label{fig2}
\end{figure}
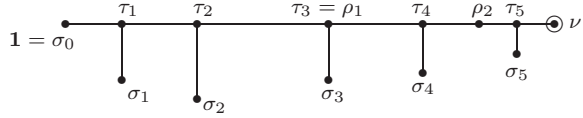

The standard (non-equivariant, minimal) dual resolution graph $\Gamma$ of the valuation $\nu$
looks like in Fig.~\ref{fig2}. The vertices
$\sigma_q$, $q=0, 1, \ldots, g$, are the dead ends
of the graph ($g$ is the number of the Puiseux pairs of a curvette corresponding to the valuation,
$\sigma_0={\bf 1}$ is the first component of the exceptional divisor), the vertices
$\tau_i$, $q=1, \ldots, g$, are the rupture points, the vertex $\nu$ corresponds to the divisorial
valuation under consideration. (The vertex $\nu$ may coincide with $\tau_g$.)
The set of vertices of the graph $\Gamma$ is ordered according to the order of the birth of
the corresponding components of the exceptional divisor. On $[\sigma_0, \nu]$
(the geodesic from $\sigma_0={\bf 1}$ to $\nu$) this order is the natural one:
$\delta_1<\delta_2$ if and only if the vertex $\delta_1$ lies on
$[\sigma_0,\delta_2]$.

The integers $m_{\sigma_q}$, $q=0,1,\ldots, g$, form the minimal set of generators of the
semigroup of values of $\nu$ and are traditionally denoted by $\overline{\beta}_q$.
One also uses the following notations.
$e_q:=\gcd(\overline{\beta}_0, \overline{\beta}_1, \ldots, \overline{\beta}_q)$,
$$
N_q:=\frac{e_{q-1}}{e_q}\,\left(=\frac{m_{\tau_q}}{m_{\sigma_q}}\right)\,.
$$

The graph $\Gamma^G$ of the minimal $G$-equivariant resolution consists of $\vert G\vert$ copies of
graph $\Gamma$ (numbered by the elements of $G$) glued together.
The gluing is defined by a sequence
$$
G=H_0\supset H_1\supset H_2\supset\ldots\supset H_k
$$
of subgroups of the group $G$ such that all $H_i$ with $i>0$ are abelian and $H_k$ is the isotropy group of
the valuation $\nu$ ($\{a\in G: a^*\nu=\nu\}$) and by a sequence by vertices $\rho_1$, \dots, $\rho_k$ of the
graph $\Gamma$ such that all of them lie on the geodesic from $\sigma_0$ to $\nu$, $\rho_1<\rho_2<\ldots<\rho_k$.
(Some of he vertices $\rho_i$ may coincide with some of the vertices $\tau_j$; the vertex $\rho_1$
may coincide with the initial vertex $\sigma_0={\mathbf{1}}$.)
The copies of $\Gamma$ numbered by the elements $a_1$ and $a_2$ from $G$ are glued along the part preceeding
$\rho_{\ell}$ (i.e., by identifying all the vertices smaller or equal to $\rho_{\ell}$) if
$a_1a_2^{-1}\in H_{\ell-1}$. (In particular the initial vertices $\sigma_0={\bf 1}$ of all the copies
are identified.) For $q=1, 2, \ldots, g$, let $j(q)$ be defined by the condition
$\rho_{j(q)}<\tau_q\le \rho_{j(q)+1}$.

For $\delta\in\Gamma^G$ (or for the corresponding $\delta\in\Gamma$), let
$M_{\delta}:=\sum_{a\in G} m_{a\delta}$. One can easily see that all the integers $M_{\delta}$, $\delta\in\Gamma$,
are different. (One has $M_{\delta_1}=M_{\delta_2}$ for $\delta_1$ and $\delta_2$ from $\Gamma^G$
if and only if there exists $a\in G$ such that $\delta_2=a\delta_1$.)
One has $M_{\tau_q}=N_qM_{\sigma_q}$.

The series $\rho P^G_{\nu}(t)$ is given by the equation
\begin{eqnarray*}
\rho P^G_{\nu}(t)&=&\prod_{q=0}^g \left(1-t^{M_{\sigma_q}}\right)^{-[G/H_{j(q)}]}\cdot
\prod_{q=1}^g \left(1-t^{N_qM_{\sigma_q}}\right)^{[G/H_{j(q)}]}\times\\
&\times&
\prod_{j=1}^{\ell} \left(1-t^{M_{\rho_j}}\right)^{[G/H_j]-[G/H_{j-1}]}\cdot
\left(1-t^{M_{\nu}}\right)^{-[G/H_k]}\,.
\end{eqnarray*}

The fact that all the integers $M_{\delta}$ are different implies that the exponents $M_{\sigma_q}$,
$q=1, \ldots, g$,
are among those which participate in the decomposition of the series $\rho P^G_{\nu}(t)$ with
negative cardinalities of the multiplicities. (The multiplicity of a binomial $(1-t^m)^{s_m}$, $s_m\in A(G)$,
is $s_m$. Its cardinality is the (virtual) number of the points of it.) It is possible that the exponents
of this sort include also $M_{\nu}$ corresponding to the divisorial valuation itself.

The subgroups $H_1\supset H_2\supset\ldots\supset H_k$ are defined by the multiplicities of all the factors
in the decomposition of the series $\rho P^G_{\nu}(t)$ into the product of the binomials.

The vertex $\sigma_0={\mathbf{1}}$ coincides with $\rho_1$ if and only if the binomial with the smallest exponent
in the decomposition of the series $\rho P^G_{\nu}(t)$
has a non-negative cardinality of the multiplicity. For $\sigma_q\le\rho_1$ one has
$M_{\sigma_q}=\vert G\vert m_{\sigma_q}$ and $M_{\rho_1}=\vert G\vert m_{\rho_1}$. These equations
give all the generators $\overline{\beta}_q$ of the semigroup of values with $\sigma_q\le\rho_1$ and also
$m_{\rho_1}$.

For $\ell\ge 1$, let $\sigma_{q(\ell)}$ be the minimal dead end greater than $\rho_{\ell}$
(i.e. there are the dead ends $\sigma_{q(\ell)}$, \dots, $\sigma_{q(\ell+1)-1}$ inbetween
$\rho_{\ell}$ and $\rho_{\ell+1}$). Let us consider the dead ends $\sigma_q$ such that
$\rho_1<\sigma_q<\rho_2$. One has
$$
M_{\sigma_{q(1)}}=\vert H_1\vert m_{\sigma_{q(1)}} + (\vert G\vert - \vert H_1\vert)m_{\rho_1}
=
\vert H_1\vert m_{\sigma_{q(1)}} + (M_{\rho_{1}} - \vert H_{1}\vert m_{\rho_{1}})\,.
$$
The smallest multiple of the exponent $M_{\sigma_{q(1)}}$ in a binomial participating in the
decomposition of the series $\rho P^G_{\nu}(t)$ is $M_{\tau_{q(1)}}=N_{q(1)}M_{\sigma_{q(1)}}$. Further,
for $\rho_1 < \sigma_{q(1)} < \sigma_{q(1)+1} < \sigma_{q(1)+2} < \cdots \sigma_{q(2)-1} < \rho_2$,
one has
\begin{eqnarray*}
M_{\sigma_{q(1)+1}}&=&\vert H_1\vert m_{\sigma_{q(1)+1}} +
(M_{\rho_{1}} - \vert H_{1}\vert m_{\rho_{1}})N_{q(1)}\,,\\
M_{\sigma_{q(1)+2}}&=&
\vert H_1\vert m_{\sigma_{q(1)}+2} +
(M_{\rho_{1}} - \vert H_{1}\vert m_{\rho_{1}})N_{q(1)}N_{q(1)+1}\,,\\
{\ }&{\ }&\ldots\\
M_{\rho_2}&=&
\vert H_1\vert m_{\rho_2} +
(M_{\rho_{1}} - \vert H_{1}\vert m_{\rho_{1}})N_{q(1)}N_{q(1)+1}\cdot\ldots\cdot N_{q(2)-1}\,.
\end{eqnarray*}
These equations
give all the generators $\overline{\beta}_q$ of the semigroup of values with $\sigma_q<\rho_2$ and also
$m_{\rho_2}$.

Assume that we have determined all the exponents $m_{\sigma_{q}}$ for $q<q(\ell)$ and also the exponent
$m_{\rho_{\ell}}$. Let us consider the dead ends $\sigma_q$ such that $\rho_{\ell}<\sigma_q<\rho_{\ell+1}$.
One has
\begin{eqnarray*}
M_{\sigma_{q(\ell)}}&=&\vert H_{\ell}\vert m_{\sigma_{q(\ell)}} +
(M_{\rho_{\ell}} - \vert H_{\ell}\vert m_{\rho_{\ell}})\,,\\
M_{\sigma_{q(\ell)+1}}&=&\vert H_{\ell}\vert m_{\sigma_{q(\ell)+1}} +
(M_{\rho_{\ell}} - \vert H_{\ell}\vert m_{\rho_{\ell}})N_{q(\ell)}\,,\\
M_{\sigma_{q(\ell)+2}}&=&\vert H_{\ell}\vert m_{\sigma_{q(\ell)+2}} +
(M_{\rho_{\ell}} - \vert H_{\ell}\vert m_{\rho_{\ell}})N_{q(\ell)}N_{q(\ell)+1}\,,\\
{\ }&{\ }&\ldots\\
M_{\rho_{\ell+1}}&=&
\vert H_{\ell}\vert m_{\rho_{\ell}} +
(M_{\rho_{\ell}} - \vert H_{\ell}\vert m_{\rho_{\ell}})N_{q(\ell)}N_{q(\ell)+1}\cdot\ldots\cdot N_{q(\ell+1)-1}\,.
\end{eqnarray*}
These equations
give all the generators $m_{\sigma_{q}}$ of the semigroup of values with $q<q(\ell+1)$ and also
$m_{\rho_{\ell+1}}$.

The described procedure recovers $m_{\sigma_q}$ for all $q\le g$.
If, in the binomials of the decomposition of the series $\rho P^G_{\nu}(t)$, there are no exponents
proportional to $M_{\sigma_g}$,
one has $\nu=\tau_g$ and the resolution graph $\Gamma$ is determined by the semigroup
$\langle\overline{\beta}_0\overline{\beta}_1,\ldots,\overline{\beta}_g\rangle$.
Otherwise the described above procedure permits to determine the exponents $m_{\rho_j}$ with
$\rho_j\ge \tau_g$ and $m_{\nu}$. This gives the $G$-equivariant resolution graph of one divisorial valuation.
\smallskip

Assume that we have a collection $\{\nu_i\}$ of divisorial valuations, $i=1,2,\ldots, r$.
To restore the equivariant resolution graph $\Gamma^G$ of the collection from the resolution graphs
of each individual valuation $\nu_i$, one has to determine the separation point $\delta_{ij}$ between each
two valuations $\nu_i$ and $\nu_j$ (for simplicity let us assume that $i=1$, $j=2$).
Let
\begin{equation}\label{two_branches}
\rho P^G_{\nu}(t_1, t_2, 1,\ldots,1)= \prod (1-t_1^{M_1}t_2^{M_2})^{s_{M_1M_2}}\,,
\end{equation}
$s_{M_1M_2}\in\Z$, be the decomposition into the product of the binomials.
The separation point $\delta_{12}$ corresponds to the maximal exponent in the decomposition (\ref{two_branches})
with
$$
\frac{M_{\delta 1}}{M_{\delta 2}}=\frac{M_{\sigma_0 1}}{M_{\sigma_0 1}}\,.
$$
This proves that the reduction $\rho P^G_{\{\nu_i\}}(\tt)\in A(G)[[t_1, \ldots, t_r]]$ of the $G$-equivariant
Poincar\'e series $P^G_{\{\nu_i\}}(\tt)$ determines the minimal
$G$-resolution graph of the set $\{\nu_i\}$ of divisorial valuations.

In order to prove that one can also determine the weak $G$-topology of the collection of valuations,
one has to show how is it possible to restore the representation of the group $G$ on $\C^2$ and the
correspondence between (some) tails of the (minimal) $G$-resolution graph and the special points on the
first component of the exceptional divisor. For that one should use the non-reduced Poincar\'e series
$P^G_{\{\nu_i\}}(\tt)\in\widetilde{A}(G)[[t_1, \ldots, t_r]]$ itself.
(If there are no special points on the first component of the exceptional divisor (this can happen only if
$G$ is cyclic), only the representation of $G$ on $\C^2$ has to be determined.)
We follow the scheme described in \cite{Ark}.

Let us consider the case of an abelian group $G$ first. If there are no special points
on the first component $E_{\bf 1}$ of the exceptional divisor, all points of $E_{\bf 1}$
are fixed with respect to the group $G$, the group $G$ is cyclic
and the representation is a scalar one. This (one dimensional)
representation is dual to the representation of the group $G$ on
the one-dimensional space generated by any linear function. The
case when there are no more components in ${\cal D}$, i.e. if the
resolution is achieved by the first blow-up, is trivial.
Otherwise let us consider a maximal component $E_{\sigma}$ among those components $E_{\tau}$
of the exceptional divisor for which $G_{\tau}=G$ and the corresponding curvette is smooth.
(The last condition can be easily detected from the resolution graph.)
The smooth part $\stackrel{\bullet}{E}_{\sigma}$ of this component contains a special point $x$
with $G_x=G$ (or all the points of $\stackrel{\bullet}{E}_{\sigma}$ are such that $G_x=G$).
The point(s) from $\stackrel{\bullet}{E}_{\sigma}$ with $G_x=G$ bring(s) into the decomposition
of the Poincar\'e series $P^G_{\{\nu_i\}}(\tt)$
the factor of the form $(1-t^M)^{-[G/G]_{\alpha}}$. The ($G$-equivariant) curvette $L$ at the described
special point of the divisor is smooth. Therefore the representation of $G$ on the
one-dimensional space generated by a $G$-equivariant equation of $L$ coincides with
the representation on the space generated by a linear function.
Let us take all factors of the form $(1-t^{\MM})^{-[G/G]_{\alpha}}$ in the decomposition of the Poincar\'e series
$P^G_{\{v_i\}}$. For each of them, the
exponent $\MM$ determines the corresponding component of the exceptional divisor and
therefore the topological type of the corresponding curvettes. The factor which corresponds to a component with a smooth curvette gives us the representation $\alpha$ on the space generated by a linear function.

Now assume that there are two special points on the first component of the resolution.
Without loss of
generality we can assume that they correspond to the coordinate axis $\{x=0\}$ and
$\{y=0\}$. The representation of the group $G$ on $\C^2$ is defined by its action on
the linear functions $x$ and $y$. For each of them this action can be recovered from
a factor of the form described above just in the same way.
Moreover, a factor, which determines the action of the group $G$ on the function $x$,
corresponds to a component of the exceptional divisor from the tail emerging from the point $\{x=0\}$.

Now let $G$ be an arbitrary (not necessarily abelian) group. For an element $g\in
G$ consider the action of the cyclic group $\langle g\rangle$ generated by $g$ on
$\C^2$. One can see that the $G$-equivariant Poincar\'e series $P^{G}_{\{v_i\}}(\tt)$
determines the $\langle g\rangle$-Poincar\'e series
$P^{\langle g\rangle}_{\{v_i\}}(\tt)$ just like in \cite[Proposition 2]{Equiv2}. This
implies that the $G$-equivariant Poincar\'e series determines the representation of the
subgroup $\langle g\rangle$. (Another way is to repeat the arguments above adjusting
them to the subgroup $\langle g\rangle$.) Therefore the $G$-Poincar\'e series
$P^{G}_{\{v_i\}}(\tt)$ determines the value of the character
of the $G$-representation on $\C^2$ for each element $g\in G$ and thus the
representation itself.
Special points of the $G$-action on the first component $E_{\bf 1}$ of the
exceptional divisor correspond to some abelian subgroups $H$ of $G$. For each such
subgroup $H$ there are two special points corresponding to different one-dimensional
representations of $H$. Again the construction above for an abelian group permits to
identify tails of the dual resolution graph with these two points.
\end{proof}

%%%%%%%%%%%%%%%%%%%%%%%%%%%%%%%%%%%%%%%%%%%%%%%%%%%%%%%%%%%%%%%%%%%%%%%%%%%%%%%%%
%%%% CURVES
%%%%%%%%%%%%%%%%%%%%%%%%%%%%%%%%%%%%%%%%%%%%%%%%%%%%%%%%%%%%%%%%%%%%%%%%%%%

Let $\{C_i\}$, $i=1, \ldots, r$, be a collection of irreducible curve singularities in
$(\C^2, 0)$ such that it does not contain curves from the same $G$-orbit and it does not contain a
smooth curve invariant with respect to a non-trivial element of $G$ whose action on $\C^2$ is not a
scalar one. Let $\{\nu_i\}$ be the corresponding collection of valuations.
Let $G_i\subset G$ be the isotropy group of the branch $C_i$, $1\le i\le r$.

\begin{theorem}\label{curve_top}
The $G$-equivariant Poincar\'e series $P^G_{\{\nu_i\}}(\tt)$ of the collection $\{\nu_i\}$ determines
the weak $G$-equivariant topology of the collection $\{\nu_i\}$ of curve valuations.
\end{theorem}

\begin{proof}
The minimal resolution graph $\Gamma$ of the plane curve singularity
$C=\bigcup\limits_{i=1}^r C_i$
is essentially the same as the graph of the divisorial valuations defined by the set of irreducible
components $\{E_{\alpha_i}\}$
of the exceptional divisor such that the strict transform of $C_i$ intersects the component
$E_{\alpha_i}$. Instead of the mark used  for the divisor $E_{\alpha_i}$
(like in Figures 1 and 2 for one valuation)
one puts an arrow corresponding to $C_i$ connected to the vertex $\alpha_i$. Notice that
there can be several arrows
connected to the same vertex, i.e. $\alpha_i=\alpha_j$ for diferent branches $C_i,C_j$.
In the case of one branch the graph looks like the one in Figure 2 but the vertex marked by $\nu$
coincides with
$\tau_{g}$ and there is an arrow connected with $\tau_g$.
The number $g$ is equal to the number of Puiseux pairs of the curve and
$m_{\sigma_i}=\bar{\beta}_i$, $0\le i\le g$, are the elements of the minimal set of generators of the
semigroup of the branch. (In particular they determine the minimal resolution graph of
the curve.)

\begin{figure}[h]
$$
\unitlength=0.50mm
\begin{picture}(120.00,40.00)(0,10)
\thinlines

\put(-40,30){\line(1,0){20}}
\put(-20,30){\circle*{2}}
\put(-20,30){\line(0,-1){15}}
\put(-23,34){{\scriptsize$\alpha_i$}}
\put(-20,30){\vector(1,1){10}}
\put(-30,5){{\scriptsize (a) $\Gamma$}}

\put(20,30){\line(1,0){20}}
\put(40,30){\circle*{2}}
\put(40,30){\line(0,-1){15}}
\put(40,30){\line(1,0){20}}
\put(60,30){\circle*{2}}
\put(57,34){{\scriptsize$\rho$}}
\put(60,30){\vector(1,1){10}}
\put(60,30){\vector(1,-1){10}}
\put(72,40){{\scriptsize $C_i$}}
\put(72,20){{\scriptsize $aC_i$}}
\put(40,5){{\scriptsize (b) $\Gamma^G$}}

\put(110,30){\line(1,0){20}}
\put(130,30){\circle*{2}}
\put(130,30){\line(0,-1){15}}
\put(130,30){\line(1,0){20}}
\put(150,30){\circle*{2}}
\put(143,24){{\scriptsize$\rho=\alpha_i$}}
\put(150,30){\vector(1,1){10}}
\put(162,40){{\scriptsize $C_i$}}
\put(120,5){{\scriptsize (c) $\Gamma$ enlarged}}

\end{picture}
$$
\caption{The graphs $\Gamma$, $\Gamma^G$ and $\Gamma$ enlarged.
}
\label{fig2}
\end{figure}
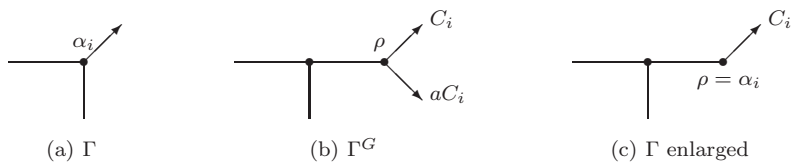

The same rules
apply for the graph $\Gamma^G$.
However $\Gamma^G$ corresponds to the
embedded resolution of the union of all the orbits of the branches of $C$. So,
it is possible that, in order to achieve the minimal equivariant resolution (i.e.
in order to separate all the conjugate
of each one of the branches $C_i$), one has to add some aditional blow-ups starting
in the point $\alpha_i$. Note that in this case some of the vertices $\rho$ (see the
notations in the proof of Theorem 1 and Figures 1 and 2) does not appear in $\Gamma$.
In order to preserve
the scheme and the notations from the proof of the case of divisorial valuations it
is better to enlarge  $\Gamma$ in such a way that the new one (also denoted by
$\Gamma$) is the minimal one in which all
the vertices $\rho$ are present (see Figure 3).
Note that $a E_{\alpha_i}=E_{a\alpha_i}$ for $a\in G$, so in this way the (new) resolution graph $\Gamma$
is just the quotient of $\Gamma^G$ by the obvious action of $G$ on $\Gamma^G$.

As in the case of divisorial valuations, for each $\delta\in \Gamma^G$
let $h_{\delta}=0$, $h_{\delta}\in \OO_{\C^2,0}$,
be the equation of a curvette at the component $E_{\delta}$,
$m_{\delta i}$ be the value $\nu_i(h_{\delta})$,
$M_{\delta i}= \sum_{a\in G} m_{(a\delta) i} = \sum_{a\in G}(a^*
\nu_i)(h_{\delta})$
and $\underbar{M}_{\delta}= (M_{\delta 1}, \ldots, M_{\delta r})\in \Z_{\ge 0}^r$.
All the $\MM_\sigma$, $\sigma\in \Gamma$, are different
and for $\sigma, \tau\in \Gamma^G$
$\MM_\sigma=\MM_\tau$
if and only if
$E_\tau = a E_\sigma$ for some $a\in G$.
Let $G_i\subset G$ be the isotropy group of the branch $C_i$, $1\le i\le r$.

For $i, j\in \{1,\ldots ,r\}$,
$m_{\alpha_i j}$ is just the intersection multiplicity between $C_i$ and $C_j$ and
$$
M_{\alpha_i j} = \sum_{a\in G} m_{(a \alpha_i) j} = \sum_{a\in G} (a^*\nu_j)(h_{\alpha_i}) =
(C_i, \bigcup_{a\in G} a C_j) =
(C_j, \bigcup_{a\in G} a C_i) =
M_{\alpha_j i}\; .
$$
In contrast with the case of divisorial valuations the projection formula
is different from the one for divisorial valuations formulated at the beginning of the proof
of Theorem 1.
Instead of it one has the following one:
For $i_0\in \{1,\ldots,r\}$ one has
\begin{equation}\label{projection}
P^{G}_{\{\nu_i\}}(\tt)_{|_{t_{i_0}=1}} =
(1-\tt^{\MM_{\alpha_{i_0}}})^{[G/G_{i_0}]_{\alpha_{i_0}}}_{|_{t_{i_0}=1}}
P^{G}_{\{\nu_i\}_{i\ne i_0}}(t_1,\ldots,t_{i_0 -1}, t_{i_0 +1}, \ldots,t_r)\,.
\end{equation}
(This can be easily deduced from (\ref{ACampo}).)
Using (\ref{projection}) repeatedly one also has:
\begin{equation}\label{projection2}
P^{G}_{\{\nu_i\}}(\tt)_{|_{t_i=1, i\neq i_0}} =
\prod_{i\neq i_0} (1- t_{i_0}^{M_{\alpha_i i_0}})^{[G/G_i]_{\alpha_i}} P^G_{\nu_{i_0}}(t_{i_0})\, .
\end{equation}

Equations (\ref{projection}) and (\ref{projection2}) imply
that in order to describe inductively the minimal
$G$-resolution graph $\Gamma^G$ one has to detect the binomial
$(1-\tt^{\MM_{\alpha_{i_0}}})$ corresponding to some $i_0$ from the $G$-equivariant Poincar\'e
series and also the intersection multiplicities of $C_{i_0}$ with the other branches
of $C$. As in the divisorial case, the necessary information about the $G$-equivariant resolution graph
can be restored from the $\rho$-reduction $\rho P^{G}_{\{\nu_i\}}(\tt)$ of the Poincar\'e series
$P^{G}_{\{\nu_i\}}(\tt)$ to the ring $A(G)[[\seq t1r]]$. From the factorization given in (\ref{ACampo})
one can write
$\rho P^G_{\{\nu_i\}}(\tt)= \prod_{\sigma\in \Gamma}
(1-\tt^{\underline{M}_{\sigma}})^{s_\sigma}$, where $s_\sigma\in A(G)$.
Note that the multiplicity $s_{\sigma}$ may be equal to zero, i.e. the binomial factor
corresponding to $\sigma$ may be absent.

The determination of the $G$-equivariant resolution graph from the series
$\rho P^{G}_{\nu}(t)$ for
one branch almost repeats the one described for one divisorial valuation, e.g. the semigroup is the same
as the one of the divisorial valuation defined by the component $E_{\tau_g}$ of the exceptional divisor.
So, let us assume $r>1$ and
let us fix $j,k\in \{1,\ldots,r\}$. The separation point
$s(\alpha_j,\alpha_k)\in
\Gamma^G$
of $\alpha_j$ and $\alpha_k$ is
defined by the condition
$[{\bf 1},\alpha_j]\cap [{\bf 1},\alpha_k]=[{\bf 1},s(\alpha_j,\alpha_k)]$. Here $[{\bf 1},\sigma]$ is
the geodesic in the dual graph $\Gamma^G$ joinning the first vertex ${\bf 1}$ with the vertex
$\sigma$. Now, let us define the separation vertex
$s(\alpha_j,k)$ of $C_j$ and $GC_k$
as the maximun of $s(\alpha_j,a \alpha_k)$ for $a\in G$. Note that, if $a\in G$ then
$s(a \alpha_j,k) = as(\alpha_j,k) \in \Gamma^G$ so
$s(j,k)=s(\alpha_j,k)$ is a well defined
vertex of the graph $\Gamma$. We
refer to it as the separation vertex of $C_i$ and $C_j$
in $\Gamma$.

The ratio $M_{\sigma j}/M_{\sigma k}$ is constant for $\sigma$ in $[{\bf 1},s(j,k)]$ and
is a strictly increasing function for $\sigma\in [s(i,j),\alpha_j]\subset \Gamma$ as well as in the
geodesic $[a s(j,k), a \alpha_j]\subset \Gamma^G$ for $a \in G$. Notice that for
$\sigma\notin \bigcup_{a\in G} \left([{\bf 1},a \alpha_j]\cup [{\bf 1},a \alpha_k] \right)$ the ratio
$M_{\sigma j}/M_{\sigma k}$ is equal to
$M_{\sigma' j}/M_{\sigma' k}$ where $\sigma'$ is the vertex such that
$$
[{\bf 1},\sigma'] = \max_{a\in G} \left\{
\left([{\bf 1},a \alpha_j]\cup [{\bf 1},a \alpha_k] \right)\cap [{\bf 1},\sigma]\right\}\,.
$$

Let $\sigma\in \Gamma$ be such that the exponent $\MM_{\sigma}$ is a maximal one
among the set of exponents $\MM_{\tau}$ appearing in the factorization
\begin{equation}\label{factor}
\rho P^{G}_{\{\nu_i\}}(\tt) = \prod_{\tau\in \Gamma\; , \; s_{\tau}\neq 0}
(1-\tt^{\MM_\tau})^{s_{\tau}} \; .
\end{equation}
(Here we use the partial order
$\MM=(\seq M1r)\le \MM'=(\seq {M'}1r)$ if and only if $M_i\le M'_i$ for all
$i=1,\ldots,r$.) Note that in this case the corresponding factor has positive
cardinality and there exists an index $j\in \{1,\ldots,r\}$ such that
$\alpha_j=\sigma$.

Let $A\subset \{1,\ldots,r\}$ be the set of indices $j$ such that
$M_{\sigma j}/M_{\sigma k}\ge M_{\tau j}/M_{\tau k}$ for all $k\in\{1,\ldots,r\}$ and
all $\tau\in \Gamma^G$ such that  the binomial $(1-\tt^{\MM_{\tau}})$ appears in
(\ref{factor}), i.e. $s_{\tau}\neq 0$.
From the comments above it is clear that all indices $j$ such that $\alpha_j=\sigma$
belong to $A$, however $A$ could contain some other indices $\ell$ such that
$\alpha_{\ell}\neq \sigma$.

Let us assume that there exists $\ell\in A$ such that
$\alpha_{\ell}\neq \sigma$. The behaviour of the ratios $M_{\tau \ell}/M_{\tau k}$
along $[1,\alpha_{\ell}]$ described above implies that $\sigma\in [1,\alpha_\ell]$.
By definition of the set $A$, for any $\tau\in [\sigma, \alpha_{\ell}]$,
$\tau\neq\sigma$, the binomial $(1-\tt^{\MM_{\tau}})$ does not appear in
(\ref{factor}), i.e. $s_{\tau}=0$, in particular
$\chi(\overset{\circ}{E}_{\tau})=0$.
As a consequence, $\alpha_{\ell}<\sigma$ and
$\alpha_{\ell}$ is the end point $\sigma_g$ on the dual graph of $C_j$ (here $j\in A$
such that $\alpha_j=\sigma$). In this case one has
$M_{\sigma \ell} < M_{\sigma j}$ and one can distinguish $\ell$ by this condition.
Note that if such an $\ell\in A$ exists then it is unique.

Let $i_0\in A$ be such that $M_{\sigma i_0}\ge M_{\sigma j}$ for all $j\in A$. Then
$\alpha_{i_0}=\sigma$ and the factor
$(1-\tt^{\MM_{\alpha_{i_0}}})^{[G/G_{i_0}]}$ appears in the factorization
(\ref{factor}). Thus, the projection formulae permits to recover the $G$-equivariant resolution graph by
induction.

As in Theorem~\ref{theo1} one has to show that the Poincar\'e series $P^{G}_{\{\nu_i\}}(\tt)$ determines
the representation of $G$ on $\C^2$, and the correspondence between ``tails" of the resolution graph.
The proof in this case does not differ
from the one made in Theorem~\ref{theo1} for
divisorial valuations since the collection $\{C_i\}$ does not contains smooth curves invariant with
respect to a non-trivial element of $G$ whose action is not a scalar one.
\end{proof}

Addresses:

A. Campillo and F. Delgado:
IMUVA (Instituto de Investigaci\'on en
Matem\'aticas), Universidad de Valladolid.
Paseo de Bel\'en, 7. 47011 Valladolid, Spain.
\newline E-mail: campillo\symbol{'100}agt.uva.es, fdelgado\symbol{'100}agt.uva.es

S.M. Gusein-Zade:
Moscow State University, Faculty of Mathematics and Mechanics, Moscow, GSP-1, 119991, Russia.
\newline E-mail: sabir\symbol{'100}mccme.ru

\end{document}